\documentclass[reqno,12pt]{amsart}

\usepackage{}
 \usepackage{mathrsfs}
  \usepackage{tensor}
\usepackage{amsfonts, amssymb,amsmath, amsthm, amsxtra, latexsym, amscd}
\usepackage[latin1]{inputenc}
\usepackage{graphicx}
\usepackage{bm}

\theoremstyle{definition}
\newtheorem*{thm*}{Theorem}

\newtheorem{thm}{Theorem}[section]
\newtheorem{cor}[thm]{Corollary}
\newtheorem{lem}[thm]{Lemma}
\newtheorem{prop}[thm]{Proposition}
\newtheorem{exam}[thm]{Example}
\newtheorem{rem}[thm]{Remark}
\newtheorem{defi}[thm]{Definition}
\numberwithin{equation}{section}

\bigskip

\def\refn#1.#2{\expandafter\def\csname#1\endcsname{[#2]}}
\def\refnr#1.{\csname#1\endcsname}

\begin{document}

\baselineskip  1.2pc

\title[Schatten class  Bergman-type and Szeg\"o-type operators ]
{Schatten class Bergman-type and Szeg\"o-type operators on  bounded symmetric domains}
\author{Lijia Ding }
\address{School of Mathematical Sciences,
Peking University, Beijing, 100086, P. R. China}
\email{ljding@pku.edu.cn}
 \subjclass[2010]{ Primary  32M15;  Secondary 32A25; 47B10}
\keywords{Bergman kernel; Szeg\"o kernel;  Singular integral operator; Schatten class; Symmetric domain.}
\thanks{ The author was  partially supported by PCPSF (2020T130016). }
\begin{abstract} In this paper, we investigate  singular integral operators induced by the Bergman kernel and  Szeg\"o kernel on the irreducible bounded symmetric domain in its standard  Harish-Chandra
realization.   We completely  characterize when  Bergman-type operators and Szeg\"o-type operators belong to Schatten class operator ideals by several  analytic  numerical invariants of the bounded symmetric domain. These results generalize a recent result on the Hilbert unit ball  due to the author and his coauthor but also cover all irreducible  bounded symmetric domains. Moreover, we obtain two trace formulae and a new integral estimate  related to the Forelli-Rudin estimate.  The key ingredient of the proofs involves the function theory  on the bounded symmetric domain and the spectrum estimate of Bergman-type and  and Szeg\"o-type operators. 

\end{abstract}
\maketitle

\section{Introduction}
 Let $\Omega$ be an irreducible bounded symmetric domain  in the standard Harish-Chandra realization, namely  $\Omega$ can be realized as the open unit ball
with respect to the so-called spectral norm in a finite dimensional complex vector space \cite{FK,Say}, thus we can identify $\Omega$ with the spectral unit ball. Denote  the normalized Lebesgue measure on  $\Omega$ by $dv$ which means that the volume measure of $\Omega$ is one. 
It is well known   \cite{FK} that  there exists a unique generic polynomial $h(z,w)$ in $z,\bar{w}$ and an analytic numerical invariant $N$ satisfying the Bergman kernel of  $\Omega$ is given by \begin{equation}\label{ker} K(z,w)=h(z,w)^{-N},\end{equation} where the invariant  $N$ is also  called the  genus of domain $\Omega.$

 Bergman-type operators are  singular integral operators induced by the (modified) Bergman kernel. In this paper, we  will mainly consider the Bergman-type integral operator in the following form. Denote  the $K$-invariant normalized measure $dv_\gamma$ on  $\Omega$ by $$dv_\gamma(w)=c_\gamma h(w,w)^{\gamma} dv(w)$$
 for $\gamma>-1,$  where $c_\gamma$ is the normalized constant and  $K$  is the isotropic subgroup in the identity
connected component of the  biholomorphic automorphism group of $\Omega.$
 \begin{defi}\label{defii}For $\alpha\in\mathbb{R}$ and $\gamma>-1,$ the Bergman-type integral operator $B_{\alpha,\gamma}$ on $L^1(dv_\gamma)$ is defined by
$$ B_{\alpha,\gamma} f(z)=\int _{\Omega}\frac{f(w)}{h(z,w)^\alpha} dv_\gamma(w),\quad z\in\Omega.$$
\end{defi}
Bergman-type integral operators play an important  role in complex analysis and operator theory. Note that $B_{N,0}$ is precisely  the standard Bergman projection on $L^2(dv).$  The boundedness of Bergman projection on $L^p(dv)$ for $1<p<\infty$ is a fundamental and open problem  on the bounded symmetric domain, which is only solved in the case of Hilbert unit ball \cite{BB,BGN,FR,Zhu3} that  is  the unit ball of rank one in the usual complex  Euclidian norm. Moreover, Bergman projection has a close relationship with  classical integral operators such as Toeplitz operators and  Hankel operators on Bergman space over  the bounded symmetric domain.  In \cite{Zhu1,Zhu2}, the  Bergman-type integral operator  was applied to  develop the holomorphic Besov function theory on the bounded symmetric domain, which was then used for  investigating  boundedness, compactness and Schatten membership \cite{Rin, Zhu4} of Hankel operators. In the case of Hilbert unit ball, the  Bergman-type integral operator can be used to study more holomorphic function spaces such as Bergman space and Bloch space, but also frequently applied to operator theory in function space \cite{ZZ,Zhu3,Zhu4}.

 Recently, the boundedness, compactness and Schatten membership of Bergman-type, Volterra-type and Toeplitz-type operators  on  the Hilbert unit ball have attracted more attention \cite{CF,CH,Ding,DiW,KU,MP,PaP,Zha}. These researches more or less base on the Forelli-Rudin asymptotic estimate of Bergman kernel on the Hilbert unit ball \cite{FR,Zhu3}. However,  in the high rank case, the Forelli-Rudin asymptotic estimate is a long time unsolved problem \cite{EnZ,FK,Yan}. The first purpose of this paper is to characterize Schatten class Bergman-type operator $B_{\alpha,\gamma}$ on the bounded symmetric domain. To overcome the main difficulty that there is no complete Forelli-Rudin asymptotic estimate in the high rank case, the author and his coauthor provide an alternative approach based on the spectrum estimate in the case of Hilbert unit ball \cite{Ding}. In this paper we will systematically develop this approach on the bounded symmetric domain. More precisely, we introduce the notation of prediagonal operator
on abstract Hilbert space and reduce the Bergman-type operator $B_{\alpha,\gamma}$ to the prediagonal operator. In our case, the prediagonal operator is normal compact and takes its eigenvalues (counting multiplicities) as its characteristic. Then combing with  asymptotic estimates of Gindikin's Gamma functions \cite{FK} and dimension of irreducible polynomial spaces occur in the Peter-Schmid-Weyl decomposition \cite{FK,Up,UpW} follows our main results (Theorem \ref{thmf} and \ref{thms}).  As its application, we obtain two trace formulae (Proposition \ref{kss}), a new integral estimate  (Corollary \ref{FRe}) related to the Forelli-Rudin estimate on the bounded symmetric domain and a compactness result (Corollary \ref{Bcpt}) of the operator $B_{\alpha,\gamma}$ from $L^\infty(dv_\gamma)$ to $L^1(dv_\gamma).$  Similarly, we also consider the Schatten class Szeg\"o-type operators which are integral operators introduced by the Szeg\"o kernel.

 These results generalize  our recent result on the Hilbert unit ball \cite[Theorem 3]{Ding}, but also covers all irreducible  bounded symmetric domains. It is worth mentioning that the condition in our results  is necessary and  sharp.  This work can be also viewed as an attempt to study the boundedness and compactness of classical integral operators without Forelli-Rudin estimate on bounded symmetric domains. Since there exists an one-to-one corresponding  between semi-simple Lie algebras of Hermitian type and    Hermitian Jordan triple Systems \cite{Say,Up,UpW}, we believe that our work should be reproduced in term of Jordan triple Systems.

 The paper is organized as follows. In Section 2, we give some preliminaries of the function theory on  bounded symmetric domains in term of the  Harish-Chandra
realization.  In section 3, we give and prove our main results on Bergman-type operators. In section 4, we prove two trace formulae  of the Bergman-type operator, as its  application, we  give an integral estimate related to the Forelli-Rudin estimate and  a compactness  result on  the Bergman-type operator $B_{\alpha,\gamma}.$ In section 5 we give our results on Szeg\"o-type operators.
 \section{ preliminaries}
 In this section, we will recall some basic results about the  bounded symmetric domain in its standard
standard Harish-Chandra realization without proofs, and the interested readers can consult \cite{FK,Hel,Say} for details.  Let $\Omega$ be an irreducible bounded symmetric domain in $\mathbb{C}^d.$ Let $G$ be the identity
connected component of the  biholomorphic automorphism group of $\Omega,$ 
  and its isotropic subgroup $K$ is given by
$$K := \{g \in G : g(o) = o\},$$
for any fixed $o\in \Omega.$
Then $\Omega =G\cdot o \approx G/K$. Denote the Lie  algebra of $G,K$ by  $\mathfrak{g,k}$ respectively. Then $\mathfrak{g}$ is semi-simple Lie algebra and $\mathfrak{k}$ has nontrivial center with the Cartan decomposition $$\mathfrak{g}=\mathfrak{k}\oplus\mathfrak{p}.$$
Chose a Cartan subalgebra $\mathfrak{h}$ of $\mathfrak{k},$ and it is also a Cartan subalgebra of  $\mathfrak{p}.$ Denote by  $\mathfrak{g}^\mathbb{C},\mathfrak{h}^\mathbb{C}$ the corresponding complexification of $\mathfrak{g},\mathfrak{h}$ respectively. A nonzero root $\gamma$ of $\mathfrak{g}^\mathbb{C}$ with respect to $\mathfrak{h}^\mathbb{C}$ is called compact if $\gamma\in\mathfrak{k}^\mathbb{C}$  and is called noncompact if $\gamma\in\mathfrak{p}^\mathbb{C}.$ Let $\theta$ be the conjugation with respect to $\mathfrak{k}\oplus i \mathfrak{p}$ and $\{e_\gamma\}$ be  a base of root vectors such that $$\theta e_\gamma=-e_{-\gamma},[e_\gamma,e_{-\gamma}]=h_\gamma,[h_\gamma,e_{\pm\gamma}]=\pm 2e_{\pm\gamma}.$$
Set $\mathfrak{p}^\pm=\sum_{\gamma} \mathbb{C}e_{\pm\gamma},$ thus we have the following canonical decomposition $$\mathfrak{g}^\mathbb{C}=\mathfrak{k}^\mathbb{C}\oplus\mathfrak{p}^+\oplus\mathfrak{p}^-.$$
Denote by $G^\mathbb{C}$  the adjointgroup of $g^\mathbb{C}.$ Let $G, K, K^\mathbb{C}, P^\pm$ be the analytic
subgroups corresponding to $\mathfrak{g}, \mathfrak{k}, \mathfrak{k}^\mathbb{C}, \mathfrak{p}^\pm, $ respectively. It follows from a  well-known lemma \cite[Lemma 4.2]{Say} that there exists a natural
(open) embedding:
$$ \Omega\approx G/K\approx GK^\mathbb{C}P^-/K^\mathbb{C}P^-\hookrightarrow P^+K^\mathbb{C}P^-/K^\mathbb{C}P^-\approx P^+\approx\mathfrak{p}^+.$$
Thus the embedding $\varphi:\Omega\rightarrow\mathfrak{p}^+$ is given by the following rule:
\begin{equation}\begin{split}
\Omega&\rightarrow\mathfrak{p}^+\notag\\
x=g\cdot o&\Leftrightarrow (\exp z)^{-1}g,
\end{split}\end{equation}
namely $\varphi$ maps $g$ to  the $K^\mathbb{C}P^-$-part of in the direct product decomposition
$P^+K^\mathbb{C}P^-.$
Due to Harish-Chandra, there exist $r$ linearly independent positive non-compact roots $\{\gamma_1,\cdots\gamma_r\}$  (relative to $\mathfrak{h}^\mathbb{C}$)
such that $\gamma_1>\gamma_2>\cdots>\gamma_r$ and $\gamma_i\pm\gamma_j$ is not a root whenever $i\neq j$ which are called strongly orthogonal roots, the positive integer $r$ is called the rank of domain $\Omega.$ Let $$e_j:=e_{\gamma_j},h_j:=h_{\gamma_j}$$ and $$e:=\sum_{j=1}^re_j.$$
For any $z\in\mathfrak{p}^+\approx\mathbb{C}^d,$ there exist $k\in K$ and $(\xi_1,\cdots,\xi_{r})\in \mathbb{R}^r$ such that
\begin{equation}\label{zexp}z=Ad(k)(\sum_{j=1}^r\xi_je_j),\end{equation}
and the $r$-tuple $(\xi_1,\cdots,\xi_r)$ is uniquely
determined by $z$ up to the order and the sign. The expression (\ref{zexp}) in fact induces the so-called spectral norm $\Vert \cdot\Vert_{sp}$ on $\mathfrak{p}^+\approx\mathbb{C}^d$ which is given by $$\Vert z\Vert_{sp}=\max_{1\leq j\leq r}\vert \xi_j\vert.$$
Indeed, it can be proved that $\varphi(\Omega)=\{z\in \mathbb{C}^d:\Vert z\Vert_{sp}<1\}.$ This is the so-called bounded  Harish-Chandra realization. Note that any two norm of the finite dimensional complex vector
space $\mathbb{C}^d$ are equivalent, thus the topologies induced by the spectral norm and the usual Euclidian norm coincide on $\mathbb{C}^d.$  Since the spectral unit ball in $\mathbb{C}^d$ is a bounded  balanced convex domain, the one-to-one correspondence between the bounded  balanced convex domains and Minkowski norm functions \cite {JaP} on $\mathbb{C}^d$ implies that the spectral norm  coincides with the  Minkowski function of the spectral unit ball.

Denote by $c =\exp i(\frac{\pi}{4})(e - \theta e)$ the Cayley transform and by ${}^{c}\mathfrak{g}$ the Lie algebra of the group ${}^cG=cGc^{-1}$. Let $\mathfrak{h}^-$ be the real span of the vectors in $i\mathfrak{h},$ then $i\mathfrak{h}^-$ is a Cartan subalgebra of the pair $({}^{c}\mathfrak{g}, {}^{c}\mathfrak{k}),$ and $\pm\frac{1}{2}(\gamma_j\pm\gamma_k),\pm\gamma_j,\pm\frac{i}{2}\gamma_j$ are roots  with respective multiplicities $a, 1,2b$ of ${}^{c}\mathfrak{g}$ with respect to $i\mathfrak{h}^-,$ where $a,b$ are independent of indexes $j, k$. Indeed, the  nonnegative integer triple $(a,b,r)$ uniquely determines an irreducible bounded symmetric domain up to biholomorphism \cite{Zhu1}. It follows from the dimension count that
$$d=r+\frac{a}{2}r(r-1)+br.$$

The Hermitian inner product induced from the Killing form $B$ and the the conjugation $\theta$ on $\mathfrak{p}^+$ is defined by
$$(z , w) = - B(z, \theta w),$$
for $z,w\in\mathfrak{p}^+.$
Let $\mathcal{P}(\mathfrak{p}^+)$ be the set of all holomorphic polynomials on $\mathfrak{p}^+,$ endowed with the $K$-invariant Fischer-Fock inner product
\begin{equation}\label{ff}(p,q)=\int_{\mathfrak{p}^+}p(z)\overline{q(z)}e^{-(z,z)}dz,\end{equation} for all $p,q\in\mathcal{P}(\mathfrak{p}^+).$ By \cite{FK,UpW} the natural action of $K$ on $\mathcal{P}(\mathfrak{p}^+)$ induces the Peter-Schmid-Weyl decomposition
$$ \mathcal{P}(\mathfrak{p}^+)=\sum_{\bm{m}\geq0}\mathcal{P}_{\bm{m}}(\mathfrak{p}^+),$$
 where $\bm{m}=( m_1,\cdots , m_r)\geq0$
runs over all integer partitions, namely $$ m_1 \geq\cdots \geq m_r\geq 0.$$ The decomposition is irreducible under the action of $K$ and is  orthogonal under the  Fischer-Fock inner product.

Let $I_\Omega$ be the set of all integer partitions with length $r.$ Define
 \begin{equation}\begin{split}
 &I_\Omega(0)=\{\bm{0}\},\notag\\
 &I_\Omega(j)=\{(m_1,\cdots,m_j,0,\cdots,0):m_1\geq\cdots\geq m_j>0\},1\leq j\leq r-1,\\
 &I_\Omega(r)=\{(m_1,\cdots,m_r):m_1\geq\cdots\geq m_r>0\}.
 \end{split}\end{equation}
 It  obvious that \begin{equation}\label{inp}I_\Omega=\bigcup_{j=0}^r I_\Omega(j)\quad \text{and} \quad I_\Omega(i)\cap I_\Omega(j)=\emptyset\end{equation} whenever $i\neq j.$
Let $ \mathfrak{a}^+$ be the real span of $\{e_1,\cdots\,e_r\},$ denote  the $K$-invariant polynomial $h$ on the real span $\mathfrak{a}^+\subset\mathfrak{p}^+$ by
$$h(\sum_{j=1}^rt_je_j)=\prod_{j=1}^r(1-t_j^2),$$
for real $r$-tuple $t=(t_1,\cdots,t_r).$ Note that $h$ is a real polynomial, thus it can be polarized, from \cite{FK}, we know that $$h(z,w)=\exp\sum z_j\frac{\partial}{\partial z_j}\exp\sum \bar{w}_j\frac{\partial}{\partial \bar{w}_j}h(z)$$ is still a polynomial in $z,\bar{w}.$ In particular, for $z$ in the form (\ref{zexp}), the following identity holds\begin{equation}\label{hxx} h(z,z)=\prod_{j=1}^r(1-\xi_j^2)\end{equation}  Denote by $A^2(dv_\gamma)$ the weighted Bergman space consisted of square integrable holomorphic functions with respect to the measure $dv_\gamma$ on $\Omega.$  The genus $N$ of $\Omega$ is defined by $$N:=2+a(r-1)+b,$$ then the Bergman kernel $K_\gamma(z,w)$ of the spectral unit ball is given by
$$ K_\gamma(z,w)=h(z,w)^{-(N+\gamma)},$$
which is degenerated to the formula (\ref{ker})  when $\gamma=0.$   The weighted Bergman space $A^2(dv_\gamma)$ is a reproducing kernel Hilbert function space, since  \begin{equation}\label{repf} f(z)=\langle f, K_{\gamma,z}\rangle_\gamma=\int_{\Omega}f(w)\overline{K_{\gamma,z}(w)}dv_\gamma(w),  \quad   z \in \Omega\end{equation}
for any $f\in A^2(dv_\gamma),$ where $$K_{\gamma,w}(z)= K_\gamma(z,w)=h(z,w)^{-(N+\gamma)}, \quad   z,w \in \Omega,$$  is also called the reproducing kernel of $A^2(dv_\gamma).$

For an integer   $r$-tuple $\bm{s}=(s_1,\cdots,s_r)\in \mathbb{Z}^r,$ the Gindikin's Gamma function \cite{FK} is given by $$\Gamma_\Omega(\bm{s})=(2\pi)^{\frac{ar(r-1)}{4}}\prod_{j=1}^r\Gamma(s_j-\frac{a}{2}(j-1)),$$
of the usual Gamma function $\Gamma$ whenever the right side is well defined. The multi-variable Pochhammer symbol \cite{FK,UpW} is
$$(\lambda)_{\bm{s}}=\frac{\Gamma_\Omega(\lambda+\bm{s})}{\Gamma_\Omega(\bm{s})}.$$
It can be verified that $$(\lambda)_{\bm{s}}= \prod_{j=1}^r(\lambda-\frac{a}{2}(j-1))_{s_j}$$ of the usual Pochhammer symbols $(\mu)_m=\prod_{j=1}^m(\mu+j-1).$
\section{Schatten class Bergman-type operators}
In this section, we will state and prove our main results on Bergman-type operators, which completely  characterize  the  Schatten class  Bergman-type operator on the  irreducible bounded symmetric domain.
Before stating  our main theoremes, we first introduce some notations involved. We define two sets $\mathscr{F}$ and $\mathscr{B}_\gamma$ of real numbers related to the bounded symmetric domain $\Omega:$
\begin{equation}\begin{split}
\mathscr{F}:&=\{f=\frac{a}{2}(l-1)-k:1\leq l\leq r,k\in\mathbb{N}\},\notag\\
\mathscr{B}_\gamma:&=\{s<N+\gamma:s\notin\mathscr{F}\},
 \end{split}\end{equation}
 where $\mathbb{N}$ is the set of all nonnegative integers. Note that $$\max\hspace{0.5mm} \mathscr{F}<\gamma+N.$$ It is easy to verity that $\alpha\in\mathscr{F}$ if and only if $(\alpha)_{\bm{m}}=0$ for all but finitely many integer partitions $\bm{m}\geq0.$

  Let $H$ be a separable Hilbert space,  denote  by $S_p(H)$ the Schatten $p$-class (or ideal) on $H$ where $0 < p<\infty.$ As we all known, Schatten classes are more refined classification of compact operators on Hilbert spaces, but also involve global estimates of the spectrum of compact operators. Moreover, $S_p(H)$ can be viewed as a noncommutative generalization of $\ell^p$ space; in particular, $S_p(H)$ will become a Banach space when provided a suitable norm for $1\leq p<\infty$. We refer the reader to  \cite{PaP,Rin,Zhu4} for details about  Schatten class on the Hilbert space. Now we can state our main results.

\begin{thm}\label{thmf} If $\alpha\in\mathscr{F},$ then the followings hold.
\begin{enumerate}
\item  The operator $B_{\alpha,\gamma}\in S_p(L^2(dv_\gamma))$ for any $p>0.$
\item   The operator $B_{\alpha,\gamma}\in S_p(A^2(dv_\gamma))$ for any $p>0.$
\end{enumerate}
\end{thm}

\begin{thm}\label{thms} Suppose $ \alpha\in\mathscr{B}_\gamma $ and $0 < p<\infty,$  then the following statements are equivalent.
\begin{enumerate}
\item $B_{\alpha,\gamma}\in S_p(L^2(dv_\gamma)).$
\item $B_{\alpha,\gamma}\in S_p(A^2(dv_\gamma)).$
\item $\widetilde{B_{\alpha,\gamma}}\in L^p(d\lambda).$
\item $p>\frac{N-1}{N+\gamma-\alpha}.$
\end{enumerate}
 \end{thm}
Where  $\widetilde{B_{\alpha,\gamma}}$ is the Berezin transform of $B_{\alpha,\gamma}$ on Bergman space $A^2(dv_\gamma)$ and $d\lambda$ is the M\"obius invariant measure on $\Omega,$ the exact definition will be given  later. The condition  $ \alpha\in\mathscr{B}_\gamma $ in Theorem \ref{thms} is necessary and sharp.
\begin{exam} For a moment, let us make our results more precisely. Due to E. Cartan \cite{Car}, there only exist six type  irreducible bounded symmetric domains up to biholomorphism, and the first four  types are also called  classical  domains. For convenience, we list their  standard Harish-Chandra realization  as follows \cite[Chapter 17]{Loo}.
\vspace{3.5mm}
\renewcommand\arraystretch{1.8}
\begin{center}$\begin{array}{|c|c|c|c|c|c|c|}\hline \text{Type}& \mathfrak{p}^+ & d & a & b & r & N \\\hline
\text{\uppercase\expandafter{\romannumeral1}}&\mathbb{C}^{1\times 1}& 1& 1&  0& 1 & 2\\\hline

\text{\uppercase\expandafter{\romannumeral1}}&\mathbb{C}^{r\times s}, r\leq s& rs\geq2& 2 &  s\!-\!r& r & r\!+\!s \\\hline

\text{\uppercase\expandafter{\romannumeral2}}&\mathbb{C}^{(2r+\epsilon) \times (2r+\epsilon)}_{asym} & r(2r\!+\!2\epsilon\!-\!1)\!\geq\!5& 4 & 2\epsilon &r & 4r\!+\!2\epsilon\!-\!2   \\\hline

\text{\uppercase\expandafter{\romannumeral3}}&\mathbb{C}^{r\times r}_{sym} &\frac{r(r+1)}{2} \! \geq \!2 & 1 & 0 & r & r\!+\!1  \\\hline

\text{\uppercase\expandafter{\romannumeral4}}&\mathbb{C}^{s\times s}_{spin} &s \!\geq\!4& s\!-\!2 & 0 & 2 & s  \\\hline

\text{\uppercase\expandafter{\romannumeral5}}&\mathbb{O}_{\mathbb{C}}^{1\times2}& 16 & 6& 4 & 2& 12 \\\hline

\text{\uppercase\expandafter{\romannumeral6}}& \mathcal{H}_3^{\mathbb{C}}(\mathbb{O})\bigotimes\mathbb{C}& 27 & 8 & 0 & 3& 18   \\\hline \end{array}$\\\end{center}
 \end{exam}

 \vspace{3.5mm}
 Where $\epsilon=0,1.$
 Thus, by direct calculation, the above Theorem \ref{thms}  can be exactly restated case-by-case  in the following table.


 \vspace{2.5mm}
\renewcommand\arraystretch{1.8}
\setlength{\arraycolsep}{2.5pt}
\begin{center}$\begin{array}{|c|c|c|c|}\hline \text{Type}& \mathscr{F}& \mathscr{B}_\gamma&B_{\alpha,\gamma}\in S_p\\\hline

\text{\uppercase\expandafter{\romannumeral1}}& \bigcup_{k=0}^\infty\{r\!-\!1\!-\!k\}  &  (-\infty,r\!+\!s)\!\setminus \! \bigcup_{k=0}^\infty\{r\!-\!1\!-\!k\} &p>\frac{r+s -1}{r+s+\gamma-\alpha }\\\hline

\text{\uppercase\expandafter{\romannumeral2}}&\bigcup_{k=0}^\infty\{2r\!-\!2\!-\!k\} &  (-\infty,4r\!+\!2\epsilon\!-\!2)\!\setminus\! \bigcup_{k=0}^\infty\{2r\!-\!2\!-\!k\} &p>\frac{4r+2\epsilon-3}{4r+2\epsilon-2+\gamma -\alpha } \\\hline

\text{\uppercase\expandafter{\romannumeral3}}& \bigcup_{k=0}^\infty\{\frac{r-1}{2}\!-\!k,\frac{r-2}{2}\!-\!k\} & (-\infty,r\!+\!1)\!\setminus\!  \bigcup_{k=0}^\infty\{\frac{r-1}{2}\!-\!k,\frac{r-2}{2}\!-\!k\} &p>\frac{r}{r+1+\gamma-\alpha }\\\hline

\text{\uppercase\expandafter{\romannumeral4}}&\bigcup_{k=0}^\infty\{\frac{s-2}{2}\!-\!k,- k\} & (-\infty,s)\!\setminus\!  \bigcup_{k=0}^\infty\{\frac{s-2}{2}-\!k,- k\} &p>\frac{s-1}{s+\gamma-\alpha }\\\hline

\text{\uppercase\expandafter{\romannumeral5}}&\bigcup_{k=0}^\infty\{3\!- \!k\}& (-\infty,12)\!\setminus\!\bigcup_{k=0}^\infty\{3\!- \!k\}&p>\frac{11}{12+\gamma-\alpha }\\\hline

\text{\uppercase\expandafter{\romannumeral6}}&\bigcup_{k=0}^\infty\{8\!- \!k\}& (-\infty,18)\!\setminus\! \bigcup_{k=0}^\infty\{8\!- \!k\}&p>\frac{17}{18+\gamma-\alpha } \\\hline \end{array}$\\\end{center}

\vspace{2.9mm}
Now we turn to prove the main theorems. We first establish  some lemmas.  The following lemma gives the regularity of the image of $B_{\alpha,\gamma}.$
\begin{lem}\label{hol} For any $ \alpha\in\mathbb{R}$ and $ \gamma>-1,$  then $B_{\alpha,\gamma} f$ is holomorphic on $\Omega$ for any $f\in L^1(dv_\gamma).$
\end{lem}
\begin{proof} It is sufficient to show that $B_{\alpha,\gamma} f$ is holomorphic on every point
of $\Omega.$ Suppose $z_0$ is an arbitrary point in $\Omega.$ Let $\Vert\cdot \Vert_{sp}$ be the spectral norm of $\Omega.$ Chose a real $r$ satisfying $\Vert z_0 \Vert_{sp}<r<1.$ Then $z_0\in B_r,$ where $ B_r=\{\Vert z\Vert_{sp}<r\}$ is an open ball with radius $r>0$ in the spectral norm. It follows from \cite[Theorem 3.8]{FK} that
$$h(z,w)^{-\alpha}=\sum_{\bm{m}\geq0}(\alpha)_{\bm{m}}K^{\bm{m}}(z,w)$$
converges uniformly and absolutely on  $ B_r\times \overline{\Omega},$ where $K^{\bm{m}}$ is the reproducing kernel of $\mathcal{P}_{\bm{m}}(\mathfrak{p}^+)$ in the Fischer-Fock inner product (\ref{ff}). Thus the dominated convergence theorem implies that
\begin{equation}\begin{split}\label{Kfz}
B_{\alpha,\gamma} f(z)&=\int_{\Omega}\frac{f(w)}{h(z,w)^\alpha}dv_\gamma(w)\\
             &=\int_{\Omega}f(w)\sum_{\bm{m}\geq0}(\alpha)_{\bm{m}}K^{\bm{m}}(z,w)dv_\gamma(w)\\
             &=\sum_{\bm{m}\geq0}(\alpha)_{\bm{m}}\int_{\Omega}f(w)K^{\bm{m}}(z,w)dv_\gamma(w),
\end{split}\end{equation}
for any $z\in B_r.$ Note that $\int_{\Omega}f(w)K^{\bm{m}}(z,w)dv_\gamma(w)$ is a holomorphic polynomial in $z$ for each $\bm{m}\geq0.$ Combing this with (\ref{Kfz}) follows that  $B_{\alpha,\gamma} f$ is holomorphic on the spectral ball $B_r.$ It leads to the desired result.
\end{proof}
 Recall that a bounded operator on a Hilbert space is called a finite rank operator if the range of the operator has finite dimension. Obviously, finite rank operators must be compact and  belong to every Schatten $p$-class with $0<p<\infty.$
 \begin{defi}Let $H$ be a separable Hilbert
space, a linear operator $T:H\rightarrow H$ is called prediagonal if there exist  an  orthogonal basis $\{e_n\}$ of $H$ and a complex number sequence $\{\lambda_n\}$ such that \begin{equation}\label{diag}  T(\sum f_ne_n)=\sum\lambda_n f_ne_n\end{equation}  whenever $\sum f_n e_n\in H$ and $\sum\lambda_n f_ne_n\in H.$
\end{defi}
In particular, $$Te_n=\lambda_n e_n$$ for each $n,$ thus the operator $T$ is densely defined.  In the following, we will frequently say that an operator $T$ is a prediagonal operator with the characteristic $\{\lambda_n\}$  if there exist  an  orthogonal basis $\{e_n\}$ and a sequence $\{\lambda_n\}$ satisfying (\ref{diag}). Note that a closed operator $T$ is  prediagonal  if and only if $T^\ast$ is prediagonal, where $T^\ast$ is the adjoint of $T;$ moreover, if  $\{\lambda_n\}$ is  the characteristic of $T$ with respect to the  orthogonal basis $\{e_n\},$ then $\{\bar{\lambda}_n\}$  is  the characteristic of $T^\ast$ and $$T^\ast e_n=\bar{\lambda}_n e_n$$ for each $n.$
When $T$ is a bounded prediagonal operator,  it can be proved that  $T$ must be normal and  the  characteristic of $T$ is  unique up to the order and the  characteristic is nothing but the sequence of eigenvalues (counting geometric multiplicities). If $dim \hspace{0.2mm}H<\infty,$ then $H$ can be isometrically embedded an infinite dimensional Hilbert space, and $T$ will be identified with its natural zero extension on the infinite dimensional Hilbert space. Moreover, in the finite dimensional case, the prediagonal operator is just a matrix in the unitary equivalent class of some diagonal matrix; in the infinite dimensional case, all normal compact operators are prediagonal as above.  The following lemma gives some criteria for boundedness, compactness  and Schatten membership of prediagonal operators by using its characteristic.
\begin{lem}\label{dgf}Let $T$ be a prediagonal operator with characteristic $\{\lambda_n\}$ as above, then the following hold.

\begin{enumerate}
\item $T$ is bounded if and only if $\{\lambda_n\}\in \ell^\infty.$
\item $T$ is compact if and only if $\{\lambda_n\}\in c_0.$
\item $T \in S_p(H)$  if and only if $\{\lambda_n\}\in \ell^p$  for $0<p<\infty.$
\end{enumerate}
\end{lem}
\begin{proof}(1) It is  obvious that if $T$ is bounded then $\{\lambda_n\}\in \ell^\infty.$  Conversely, for any $f=\sum f_n e_n\in H,$ since $\{\lambda_n\}\in \ell^\infty,$ Parseval equality implies that $$ \Vert \sum \lambda_nf_ne_n\Vert\leq\Vert \{\lambda_n\}\Vert_\infty\Vert f\Vert.$$ Thus  $$\Vert T f\Vert = \Vert \sum \lambda_nf_ne_n\Vert\leq\Vert \{\lambda_n\}\Vert_\infty\Vert f\Vert.$$ Since $f$ is arbitrary, it implies that $T$ is bounded.

(2) Suppose that $T$ is compact. Note that $e_n\rightarrow0$ weakly as $n\rightarrow0.$  Combing with the well-known fact that a compact
operator maps a weakly convergent sequence into a strongly convergent one, we thus obtain that $$\lim \vert\lambda_n\vert=\lim \Vert Te_n\Vert=0,$$
namely $\{\lambda_n\}\in c_0.$ Conversely, suppose  $\{\lambda_n\}\in c_0.$ By (1) we know that $T$ is bounded. Define a finite rank operator sequence $\{T_k\}$ on $H$ by
$$T_k f=\sum_{j=1}^k\lambda_j\langle f, e_j\rangle e_j,$$ for any $f=\sum f_n e_n\in H.$ Since  $\{\lambda_n\}\in c_0,$ it is easy to see that $T_k\rightarrow T$ in the operator norm. Thus $T$ is compact.

(3) It suffices to consider the case that $T$ is a normal compact operator. Denote the point spectrum of $T$ by $\sigma(T).$ Then $\{\lambda_n:n\in\mathbb{N}\}\subset\sigma(T)$ as set. If there exists $\lambda\in\sigma(T)$ but $\lambda\notin\{\lambda_n\},$ then there exists a nonzero $e\in H$ satisfying $Te=\lambda e$ and $e\bot\{e_n\}.$ Note that $\{e_n\}$ is an orthogonal basis of $H$, it implies that $e=0,$ a contradiction with $e\neq0.$ Thus $\{\lambda_n\}=\sigma(T)$ as set. Now suppose  $\lambda'\in\{\lambda_n\}=\sigma(T).$ Note that $M(\lambda')\leq dim\hspace{0.5mm} Ker(\lambda'-T),$ which is the geometric multiplicity of the nonzero eigenvalue $\lambda'.$ Then the multiplicity $M(\lambda')$ of $\lambda'$ in $\{\lambda_n\}$ must be finite. We turn to prove that $M(\lambda')=dim\hspace{0.5mm} Ker(\lambda'-T).$   If $M(\lambda')<dim\hspace{0.5mm} Ker(\lambda'-T),$ then there exists a nonzero $e'\in H$ such that $Te'=\lambda' e'$ and $e'\bot\{e_n\},$ this is a  contradiction since  $\{e_n\}$ is an orthogonal basis. It prove that $M(\lambda')=dim\hspace{0.5mm} Ker(\lambda'-T).$ Thus $\{\lambda_n\}=\sigma(T)$ as set and the multiplicity $M(\lambda_n)$  coincides with its  geometric multiplicity for any nonzero $\lambda_n.$ On the other hand, $T$ is normal compact, it follows from the spectral theorem and functional calculus for compact normal operators that  $\{\vert \lambda_n\vert \}=\sigma(\vert T\vert)$ as set and the multiplicity $M(\vert \lambda_n\vert)$  coincides with its  geometric multiplicity for any nonzero $\vert \lambda_n\vert.$ Thus $T \in S_p(H)$  if and only if $\{\lambda_n\}\in \ell^p$  for any $0<p<\infty.$
\end{proof}
\begin{lem}\label{bksd} The operator $B_{\alpha,\gamma}:L^2(dv_\gamma)\rightarrow L^2(dv_\gamma)$ is bounded if and only if $\alpha\leq N+\gamma.$
\end{lem}
\begin{proof} Clearly that $B_{N+\gamma,\gamma}:L^2(dv_\gamma)\rightarrow A^2(dv_\gamma)$ is the Bergman orthogonal projection and is bounded. In what follows, we show that the operator identity \begin{equation}\label{oid}B_{\alpha,\gamma}B_{N+\gamma,\gamma} =B_{\alpha,\gamma}\end{equation} holds on $ L^2(dv_\gamma)$ for any $\alpha\in \mathbb{R}.$ It suffices to show that $$B_{\alpha,\gamma}(Id-B_{N+\gamma,\gamma})=0$$ on $L^2(dv_\gamma).$
 Since Bergman space $A^2(dv_\gamma)$ is a closed subspace of the Hilbert space $L^2(dv_\gamma),$
 together with the orthogonal projection theorem, it suffices to show that
  \begin{equation}\label{ka0} B_{\alpha,\gamma}|_{(A^2(dv_\gamma))^\bot}=0,\end{equation} where $(A^2(dv_\gamma))^\bot$ is the orthogonal complement of Bergman space $A^2(dv_\gamma)$ in $ L^2(dv_\gamma).$ Suppose $f\in (A^2(dv_\gamma))^\bot,$ namely $f\in L^2(dv_\gamma)$ and $\langle f,g\rangle=0$ for any $g\in A^2(dv_\gamma);$ in particular, \begin{equation}\label{fp}\langle f,P\rangle=0\end{equation} for any holomorphic polynomial $P\in \mathcal{P}(\mathfrak{p}^+)$. Now we prove that $B_{\alpha,\gamma} f=0$ for any $f\in (A^2(dv_\gamma))^\bot\subset L^2(dv_\gamma).$ It implies from Lemma \ref{hol} that $B_{\alpha,\gamma} f$ is holomorphic on $\Omega.$ By the uniqueness theorem of holomorphic  functions, it suffices to show that $B_{\alpha,\gamma} f=0$ on a  spectral ball $B_r$ with $0<r<1.$
Then (\ref{Kfz}) and (\ref{fp}) yield  that
\begin{equation}\begin{split}\label{kfff}
B_{\alpha,\gamma} f(z)&=\int_{\Omega}\frac{f(w)}{h(z,w)^\alpha}dv_\gamma(w)\\
                                &=\sum_{\bm{m}\geq0}(\alpha)_{\bm{m}}\int_{\Omega}f(w)K^{\bm{m}}(z,w)dv_\gamma(w)\\
                                &=\sum_{\bm{m}\geq0}(\alpha)_{\bm{m}}\langle f,K^{\bm{m}}_z\rangle_\gamma\\
                          & =0,
 \end{split}\end{equation}
for any $z\in B_r.$ The last equality in (\ref{kfff}) holds, since  $K^{\bm{m}}_z\in \mathcal{P}(\mathfrak{p}^+)$ for any fixed $z\in B_r$ and $\bm{m}\geq0.$ Thus the operator identity (\ref{oid}) holds. Combing Lemma \ref{hol} and  the operator identity (\ref{oid}) yields that $B_{\alpha,\gamma}:L^2(dv_\gamma)\rightarrow L^2(dv_\gamma)$ is bounded if and only if $B_{\alpha,\gamma}:A^2(dv_\gamma)\rightarrow A^2(dv_\gamma)$ is bounded. In the following, we prove that $B_{\alpha,\gamma}:A^2(dv_\gamma)\rightarrow A^2(dv_\gamma)$ is a prediagonal operator. Suppose $f=\sum_{\bm{m}\geq0}f_{\bm{m}}\in A^2(dv_\gamma)$ with the Peter-Schmid-Wely decomposition of $A^2(dv_\gamma).$ By    Schur's lemma,
     the ratio of any two $K$-invariant inner
 products is a constant on the irreducible space $\mathcal{P}_{\bm{m}}(\mathfrak{p}^+),$  in fact, by \cite[Corollary 3.7]{FK} we know that \begin{equation}\label{bff}\frac{\langle p, q\rangle_\gamma}{(p,q)}=\frac{1}{(N+\gamma)_{\bm{m}}},\end{equation}  for any $p,q\in \mathcal{P}_{\bm{m}}(\mathfrak{p}^+).$ Since $K$ is a compact group, it follows that the orbit $K\cdot z\subset \Omega$ is compact for any $z\in \Omega.$ Then, combing this with the
 integral formula (1.11) in \cite{FK} implies that
 \begin{equation}\begin{split}\label{kf0}
B_{\alpha,\gamma} f(z)
                                &=\sum_{\bm{m}\geq0}(\alpha)_{\bm{m}}\int_{\Omega}\sum_{\bm{l}\geq0}f_{\bm{l}}(w) K^{\bm{m}}(z,w)dv_\gamma(w)\\
                                &=\sum_{\bm{m}\geq0}c_\gamma \cdot(\alpha)_{\bm{m}}\int_0^1\cdots\int_0^1\prod t_j^{2b+1}(1-t_j^2)^\gamma \prod_{l<k}\vert t_l^2-t_k^2\vert^a dt_1\cdots dt_r\\
                                &\hspace{1cm}\times\int_{K}\sum_{\bm{n}\geq0}f_{\bm{n}}(g\sum_{j=1}^rt_je_j) K^{\bm{m}}(z,g\sum_{j=1}^rt_je_j)dg\\
                                &=\sum_{\bm{m}\geq0}c_\gamma c\cdot(\alpha)_{\bm{m}}\int_0^1\cdots\int_0^1\prod t_j^{2b+1}(1-t_j^2)^\gamma\prod_{l<k}\vert t_l^2-t_k^2\vert^adt_1\cdots dt_r\\
                                &\hspace{1cm}\times\sum_{\bm{n}\geq0}\int_{K}f_{\bm{n}}(g\sum_{j=1}^rt_je_j) K^{\bm{m}}(z,g\sum_{j=1}^rt_je_j)dg\\
                                &=\sum_{\bm{m}\geq0}c_\gamma c\cdot(\alpha)_{\bm{m}}\int_0^1\cdots\int_0^1\prod t_j^{2b+1}(1-t_j^2)^\gamma\prod_{l<k}\vert t_l^2-t_k^2\vert^adt_1\cdots dt_r\\
                                &\hspace{1cm}\times\int_{K}f_{\bm{m}}(g\sum_{j=1}^rt_je_j) K^{\bm{m}}(z,g\sum_{j=1}^rt_je_j)dg\\
                                &=\sum_{\bm{m}\geq0}(\alpha)_{\bm{m}}\langle f_{\bm{m}}, K^{\bm{m}}_z\rangle_\gamma\\
                                &=\sum_{\bm{m}\geq0}\frac{(\alpha)_{\bm{m}}}{(N+\gamma)_{\bm{m}}}f_{\bm{m}}(z),\\
 \end{split}\end{equation}
 where $c$ is an integral constant.
It implies from (\ref{kf0}) that $B_{\alpha,\gamma}$ is a prediagonal operator with characteristic $\{\frac{(\alpha)_{\bm{m}}}{(N+\gamma)_{\bm{m}}}\}$ on $A^2(dv_\gamma),$ where of cause counting  multiplicities. Thus by Lemma \ref{dgf}, $ B_{\alpha,\gamma}$ is bounded on $A^2(dv_\gamma)$ if and only if  $\{\frac{(\alpha)_{\bm{m}}}{(N+\gamma)_{\bm{m}}}\}\in \ell^\infty.$ By the definition of the multi-variable Pochhammer symbol, we obtain that
  \begin{equation}\begin{split}\label{poc}
  \frac{(\alpha)_{\bm{m}}}{(N+\gamma)_{\bm{m}}}&=\frac{\Gamma_\Omega(\alpha+\bm{m})}{\Gamma_\Omega(N+\gamma+\bm{m})}\frac{\Gamma_\Omega(N+\gamma)}{\Gamma_\Omega(\alpha)}\\
                                        &=\frac{\Gamma_\Omega(N+\gamma)}{\Gamma_\Omega(\alpha)}\prod_{j=1}^r\frac{\Gamma(\alpha+m_j-(j-1)\frac{a}{2})}{\Gamma(N+\gamma+m_j-(j-1)\frac{a}{2})}.
   \end{split}\end{equation}
   Then Stirling's formula and (\ref{poc}) imply that  \begin{equation}\label{asyy}\left\vert\frac{(\alpha)_{\bm{m}}}{(N+\gamma)_{\bm{m}}}\right\vert\sim \prod_{j=1}^k\frac{1}{m_j^{N+\gamma-\alpha}},\end{equation} as $I_\Omega(k)\ni\bm{m}\rightarrow\infty,$ for $1\leq k\leq r.$
 The notation $A(\bm{m})\sim B(\bm{m})$ as $I_\Omega(k)\ni\bm{m}\rightarrow\infty$
 means that the ratio $\frac{A(\bm{m})}{B(\bm{m})}$ has a positive finite limit as $$m_1\rightarrow \infty,\cdots,m_k\rightarrow \infty$$  for $1\leq k\leq r.$ Consequently, $\{\frac{(\alpha)_{\bm{m}}}{(N+\gamma)_{\bm{m}}}\}\in \ell^\infty$ if and only if $\alpha\leq N+\gamma.$ It finishes the proof.
\end{proof}
\begin{rem} In the case of Hilbert unit ball, the operator identity  (\ref{oid}) holds on $L^p(dv_\gamma)$ for any $1<p<\infty,$ since the  Bergman projection $B_{N,\gamma}$ is bounded on $L^p(dv_\gamma)$ for any $1<p<\infty;$ see \cite{CF,DiW}. However, in the high rank case, the boundedness of   Bergman projection $B_{N+\gamma,\gamma}$ on $L^p(dv_\gamma)$ for all $1<p<\infty$ is not valid; see \cite{BB,BGN}.
\end{rem}

We have proved that $B_{\alpha,\gamma}:A^2(dv_\gamma)\rightarrow A^2(dv_\gamma)$ is a prediagonal operator. Note that  $A^2(dv_\gamma)$ is a closed subspace of the Hilbert space $L^2(dv_\gamma)$ and $B_{\alpha,\gamma}(A^2(dv_\gamma))^\bot=0,$ it follows that $B_{\alpha,\gamma}:L^2(dv_\gamma)\rightarrow L^2(dv_\gamma)$ is prediagonal. In Definition \ref{defii}, $B_{\alpha,\gamma}$ is only defined on $L^1(dv_\gamma)$; however,  $B_{\alpha,\gamma}$ can be defined on $L^1(dv_\beta)$ for any $\beta>-1.$ Denote by $H_\lambda\subset L^2(dv_\beta)$ the solution space of the following linear integral equation
 $$ \lambda f(z)-\frac{c_\gamma}{c_\beta}\int_{\Omega}h(z,z)^{\gamma-\beta}h(z,w)^{-\alpha}f(w)dv_\beta(w)=0,$$ for parameter $\lambda\in\mathbb{C}.$  It is  obvious that $H_\lambda\bot H_\mu$ if $\lambda\neq\mu,$ and there only exist countably many $\lambda\in\mathbb{C}$ satisfying $H_\lambda\neq\emptyset.$
 Then Fubini's Theorem and the previous discussion of prediagonality between a closed operator and its adjoint operator imply the following.

\begin{cor} The operator  $B_{\alpha,\gamma}:L^2(dv_\beta)\rightarrow L^2(dv_\beta)$ is prediagonal if and only if $$L^2(dv_\beta)=\bigoplus_{\lambda\in\mathbb{C}}H_\lambda.$$
\end{cor}

Combing  the definition of $\mathscr{F}$ with (\ref{kf0}) yields the following conclusion.
 \begin{prop}\label{fir}The operator $B_{\alpha,\gamma}:L^2(dv_\gamma)\rightarrow L^2(dv_\gamma)$ is finite rank operator if and only if $\alpha\in\mathscr{F}.$
 \end{prop}
It is easy to see that Theorem \ref{thmf} is a direct corollary of Lemma \ref{hol} and  Proposition \ref{fir}. Denote by $\sigma(B_{\alpha,\gamma},A^2(dv_\gamma))$ the point spectrum (the collections of eigenvalues) of $B_{\alpha,\gamma}$ on the Bergman
space $A^2(dv_\gamma).$
\begin{lem}\label{poisp} For  $\alpha\in\mathbb{R},$ then $$\sigma(B_{\alpha,\gamma},A^2(dv_\gamma))=\left\{\frac{(\alpha)_{\bm{m}}}{(N+\gamma)_{\bm{m}}}:\bm{m}\geq0\right\}.$$
\end{lem}
\begin{proof} Denote $\mu_{\bm{m}}=\frac{(\alpha)_{\bm{m}}}{(N+\gamma)_{\bm{m}}}$ for all $\bm{m}\geq0.$ By formula (\ref{kf0}), it implies that
 $$\left\{\mu_{\bm{m}}:\bm{m}\geq0\right\}\subset \sigma(B_{\alpha,\gamma},A^2).$$ Thus, it suffices to prove that $$ \sigma(B_{\alpha,\gamma},A^2(dv_\gamma))\subset \left\{\mu_{\bm{m}}:\bm{m}\geq0\right\}.$$
 Suppose $\mu\in\sigma(B_{\alpha,\gamma},A^2(dv_\gamma)),$ then there exists a nonzero
$f \in A^2(dv_\gamma)$ such that
\begin{equation}\label{chk}B_{\alpha,\gamma} f=\mu f.\end{equation}
 Let $E_{\bm{m}}=\{e_{\bm{m}}^{(j)}:j=1,\cdots,d_{\bm{m}}\}$ be an orthogonal basis of $\mathcal{P}_{\bm{m}}( \mathfrak{p}^+)$ under the Bergman inner product $\langle\cdot,\cdot\rangle_\gamma$ for all $\bm{m}\geq0,$ where $d_{\bm{m}}=dim\hspace{1mm}\mathcal{P}_{\bm{m}}(\mathfrak{p}^+).$ It is easy to see that $\cup_{\bm{m}\geq0}E_{\bm{m}}$ is an orthogonal basis of the Bergman space $A^2(dv_\gamma).$ Thus $f$ has the following representation
 \begin{equation}\label{rep}f=\sum_{\bm{m}\geq0}\sum_{j=1}^{d_{\bm{m}}}\langle f,e_{\bm{m}}^{(j)}\rangle_\gamma\hspace{0.5mm}e_{\bm{m}}^{(j)}.\end{equation} Combing this representation with (\ref{Kfz}), (\ref{bff}) and (\ref{chk}), we obtain that  \begin{equation}\label{repp}\sum_{\bm{m}\geq0}\sum_{j=1}^{d_{\bm{m}}}(\mu-\mu_{\bm{m}})\langle f,e_{\bm{m}}^{(j)}\rangle_\gamma\hspace{0.5mm}e_{\bm{m}}^{(j)}=0  \end{equation} as function on the spectral ball $B_r$ with $0<r<1.$ Note that  $B_r$ is biholomorphic to $\Omega=B_1$ for any $0<r<1.$  This along with formula (1.11) in \cite{FK}, it  is easy to see that $$\bigcup_{\bm{m}\geq0} r^{-(d+\vert \bm{m}\vert)}\sqrt{\frac{(N)_{\bm{m}}}{(N+\gamma)_{\bm{m}}}}E_{\bm{m}}$$ is an  orthogonal basis of $A^2(B_r,dv).$ The reproducing property (\ref{repf}), Cauchy's inequality and (\ref{hxx}) imply that convergences in (\ref{rep}) and (\ref{repp}) are uniform in the compact subset. Thus we have

 \begin{equation}\begin{split}\label{rmfe}
r^{2(d+\vert \bm{m}\vert)}&\frac{(N+\gamma)_{\bm{m}}}{(N)_{\bm{m}}}(\mu-\mu_{\bm{m}'})\langle f,e_{\bm{m}'}^{(j)}\rangle_\gamma\\
 &=\int_{B_r} (\mu-\mu_{\bm{m}'})\langle f,e_{\bm{m}'}^{(j)}\rangle_\gamma\hspace{0.5mm} e_{\bm{m}'}^{(j)}\overline{e_{\bm{m}'}^{(j)}}dv\\
& =\sum_{\bm{m}\geq0}\sum_{j=1}^{d_{\bm{m}}}\int_{B_r} (\mu-\mu_{\bm{m}})\langle f,e_{\bm{m}}^{(j)}\rangle_\gamma\hspace{0.5mm} e_{\bm{m}'}^{(j)} \overline{e_{\bm{m}'}^{(j)}}dv\\
& =\int_{B_r}\sum_{\bm{m}\geq0}\sum_{j=1}^{d_{\bm{m}}}(\mu-\mu_{\bm{m}})\langle f,e_{\bm{m}}^{(j)}\rangle \hspace{0.5mm}  e_{\bm{m}}^{(j)} \overline{e_{\bm{m}'}^{(j)}}dv\\
&=0,
    \end{split}\end{equation}
    for any $\bm{m}'\geq0.$
 Since $f$ is nonzero,  (\ref{rep}) implies that there exist $\bm{m}_0\geq0$ and $j_0$ such that \begin{equation}\label{repc}\langle f,e_{\bm{m}_0}^{(j_0)}\rangle\neq0 .\end{equation}
 Then (\ref{rmfe}) and (\ref{repc}) yield that $\mu=\mu_{\bm{m}_0}$ and \begin{equation}\label{eigp} f=\sum_{j=1}^{d_{\bm{m}_0}}\langle f,e_{\bm{m}_0}^{(j)}\rangle \hspace{0.5mm}e_{\bm{m}_0}^{(j)}.\end{equation}
 It completes the proof.
\end{proof}
\begin{rem} \label{preka} Indeed, (\ref{kf0}), (\ref{chk}), (\ref{rep}) and (\ref{eigp}) imply that $B_{\alpha,\gamma}: A^2(dv_\gamma)\rightarrow A^2(dv_\gamma) $ is the prediagonal operator with point spectrum (counting geometric multiplicities) as its characteristic whether  $B_{\alpha,\gamma}$ is bounded or not.
\end{rem}

\begin{lem}\label{dime}
(1) There exists an uniform positive constant $C$ satisfying \begin{equation}\label{dim}
dim\hspace{1mm}\mathcal{P}_{\bm{m}}(\mathfrak{p}^+)\leq C\prod_{j=1}^k m_j^{(r-j)a+b}\end{equation}
for any $\bm{m}\in I_\Omega(k),1\leq k\leq r.$

(2)\cite{UpW} Moreover, there exists an uniform  positive constant $C$ satisfying  \begin{equation}\label{}
 \frac{1}{C} m_1^{(r-1)a+b} \leq dim\hspace{1mm}\mathcal{P}_{\bm{m}}(\mathfrak{p}^+)\leq C m_1^{(r-1)a+b}\end{equation}
for any $\bm{m}\in I_\Omega(1).$
\end{lem}
\begin{proof}(1) For any $\beta,\gamma\in\mathbb{C}$ and $k\in \mathbb{N},$ we define $(\beta)_{k,\gamma}$ by
            $$(\beta)_{k,\gamma}=\prod_{j=0}^{k-1}(\beta+j\cdot\gamma).$$
            Obviously, $(\beta)_{k,1}=(\beta)_{k}$ for any  $\beta\in\mathbb{C}$ and $k\in \mathbb{N}.$ Now suppose  $\bm{m}\in I_\Omega(k),1\leq k\leq r.$ By \cite[Lemma 2.6 and 2.7]{Up} it follows that
           \begin{equation}\begin{split}
           dim\hspace{1mm}\mathcal{P}_{\bm{m}}(\mathfrak{p}^+)&=\frac{(\rho)_{\bm{m}}}{(\rho-b)_{\bm{m}}}\prod_{1\leq i<j\leq r}\frac{m_i-m_j+\frac{a}{2}(j-i)}{\frac{a}{2}(j-i)}\cdot\frac{(m_i-m_j+\frac{a}{2}(j-i-1))_{a-1}}{(1+\frac{a}{2}(j-i-1))_{a-1}}\notag\\
            &\overset{\text{def}}{=} \frac{(\rho)_{\bm{m}}}{(\rho-b)_{\bm{m}}}I(a,\bm{m}),
           \end{split}\end{equation}
where $\rho=1+\frac{a}{2}(r-1)+b.$ By the definition and Stirling's formula, we then get that  \begin{equation}\begin{split}\label{rhom}
\frac{(\rho)_{\bm{m}}}{(\rho-b)_{\bm{m}}}&=\frac{\Gamma_{\Omega}(\rho+\bm{m})}{\Gamma_{\Omega}(\rho-b+\bm{m})}\frac{\Gamma_{\Omega}(\rho-b)}{\Gamma_{\Omega}(\rho)}\\
                                         &=\frac{\Gamma_{\Omega}(\rho-b)}{\Gamma_{\Omega}(\rho)}\prod_{j=1}^{r}\frac{\Gamma(\rho+m_j-\frac{a}{2}(j-1))}{\Gamma(\rho-b+m_j-\frac{a}{2}(j-1))}\\
                                         &\sim \prod_{j=1}^{k}m_j^{b},
 \end{split}\end{equation}
as $I_\Omega(k)\ni\bm{m}\rightarrow\infty.$   On the other hand,
\begin{equation}\begin{split}\label{iam}
I(a,\bm{m})&=\prod_{1\leq i<j\leq r}\frac{m_i-m_j+\frac{a}{2}(j-i)}{\frac{a}{2}(j-i)}\cdot \prod_{1\leq i<j\leq r}\frac{(m_i-m_j+\frac{a}{2}(j-i-1))_{a-1}}{(1+\frac{a}{2}(j-i-1))_{a-1}}\\
           &=\left(\prod_{j=1}^k m_j^{r-j}\prod_{1\leq i<j\leq k}\frac{1-\frac{m_j}{m_i}+\frac{\frac{a}{2}(j-i)}{m_i}}{\frac{a}{2}(j-i)}\right)\cdot\prod_{1\leq i\leq k<j\leq r}\frac{1+\frac{\frac{a}{2}(j-i)}{m_i}}{\frac{a}{2}(j-i)}\\
           &\hspace{0.5cm}\times\left( \prod_{j=1}^k m_j^{(r-j)(a-1)}\prod_{1\leq i<j\leq k}\frac{(1-\frac{m_j}{m_i}+\frac{\frac{a}{2}(j-i-1)}{m_i})_{a-1,\frac{1}{m_i}}}{(1+\frac{a}{2}(j-i-1))_{a-1}}\right)\\
           &\hspace{0.5cm}\times\prod_{1\leq i\leq k<j\leq r}\frac{(1+\frac{\frac{a}{2}(j-i-1)}{m_i})_{a-1,\frac{1}{m_i}}}{(1+\frac{a}{2}(j-i-1))_{a-1}}\cdot\prod_{k+1\leq i<j\leq r}\frac{(\frac{a}{2}(j-i-1))_{a-1}}{(1+\frac{a}{2}(j-i-1))_{a-1}}\\
           &\leq C_k\prod_{j=1}^k m_j^{(r-j)a}.
 \end{split}\end{equation}
Combing (\ref{rhom})  with (\ref{iam}), we thus obtain
$$ dim\hspace{1mm}\mathcal{P}_{\bm{m}}(\mathfrak{p}^+)=\frac{(\rho)_{\bm{m}}}{(\rho-b)_{\bm{m}}}I(a,\bm{m})\leq C\prod_{j=1}^k m_j^{(r-j)a+b},$$
for any $\bm{m}\in I_\Omega(k),1\leq k\leq r,$ where $C=\max_{1\leq k\leq r}\{C_k\}.$

(2) Now suppose $\bm{m}=(m_1,0,\cdots,0)\in I_\Omega(1).$ The direct calculation implies that \begin{small}$$I(a,\bm{m})=m_1^{(r-1)a}\prod_{j=2}^r\frac{1+\frac{\frac{a}{2}(j-1)}{m_1}}{\frac{a}{2}(j-1)}\cdot\frac{(1+\frac{a}{2m_1}(j-2))_{a-1,\frac{1}{m_1}}}{(1+\frac{a}{2}(j-2))_{a-1}}\prod_{2\leq i<j\leq r}\frac{(\frac{a}{2}(j-i-1))_{a-1}}{(1+\frac{a}{2}(j-i-1))_{a-1}}.$$\end{small}
This along with (\ref{rhom}) yields the desired result.
\end{proof}
\begin{rem} In the case of Hilbert unit ball, the $I(\Omega)$ is just the set of nonnegative integers and $\mathcal{P}_m(\mathfrak{p}^+)$ is just  homogeneous polynomial space with degree $m.$ In this case, the dimension of $\mathcal{P}_m(\mathfrak{p}^+)$ is  $$dim \hspace{0.3mm}\mathcal{P}_m(\mathfrak{p}^+)= \frac{(m+1)_{d-1}}{(d-1)!}, \quad m\geq 0,$$ which coincides with (2) of Lemma \ref{dime}.
\end{rem}
Now we recall the definition of Berezin transform related to the Bergman space $A^2(dv_\gamma)$ on the bounded symmetric domain $\Omega.$
Recall that Bergman space $A^2(dv_\gamma)$ is a reproducing kernel Hilbert function space, whose reproducing kernel is given by  $$K_{\gamma,w}(z)= K_\gamma(z,w)=h(z,w)^{-(N+\gamma)}, \quad   z,w \in \Omega.$$
The normalized reproducing kernel of  $A^2(dv_\gamma)$ is
 \begin{equation}\label{nrep}  k_{\gamma,w}(z)=\frac{K_\gamma(z,w)}{\sqrt{K_\gamma(w,w)}}=\frac{h(w,w)^\frac{N+\gamma}{2}}{h(z,w)^{N+\gamma}},  \quad   z,w \in \Omega.\end{equation}
For a linear operator $T$ on $A^2(dv_\gamma)$, the Berezin transform $ \widetilde{T}$ of $T$ is given by
\begin{equation}\label{bert}  \widetilde{T}(z)=\langle T k_{\gamma,z},k_{\gamma,z}\rangle_\gamma,\quad z \in \Omega.\end{equation}
The M\"obius  invariant measure $d\lambda$ on $\Omega $ is defined by $$ d\lambda(z)=\frac{dv(z)}{h(z,z)^{N}}.$$ It is easy to verify that the measure $d\lambda$  is biholomorphic invariant.
The Berezin transform  is an important tool in the operator theory on the holomorphic function space on the Hilbert ball; see \cite{Zhu3,Zhu4} for more details.
In what follows, we exactly calculate the Berezin transform of $B_{\alpha,\gamma}$ on $\Omega.$
\begin{lem}\label{bere} For any $\alpha\in\mathbb{R},$
$$ \widetilde{B_{\alpha,\gamma}}(z)=h(z,z)^{N+\gamma-\alpha},$$
on $\Omega.$
\end{lem}
\begin{proof}For any fixed $z\in \Omega,$ \cite[Theorem 3.8]{FK} implies that $h(w,z)^{-\alpha}$ is a bounded holomorphic function on $\Omega,$ in particular, $h(w,z)^{-\alpha}\in A^2(dv_\gamma)$ for any $\alpha\in\mathbb{R}.$   Then (\ref{repf}) implies that
\begin{equation}\begin{split}\label{kkz}
B_{\alpha,\gamma} K_{\gamma,z}(w)&=\int_{\Omega}\frac{1}{h(w,u)^\alpha}\frac{1}{h(u,z)^{N+\gamma}}dv_\gamma(u)\\
               &=\overline{\int_{\Omega}\frac{1}{h(u,w)^\alpha}\frac{1}{h(z,u)^{N+\gamma}}dv_\gamma(u)}\\
               &=\overline{\langle h^{-\alpha}_w,K_{\gamma,z} \rangle_\gamma}\\
               &=h^{-\alpha}(w,z).
 \end{split}\end{equation}
Combing this with (\ref{nrep}) and (\ref{bert}) shows that
 \begin{equation}\begin{split}
 \widetilde{B_{\alpha,\gamma}}(z)&=\langle B_{\alpha,\gamma} k_{\gamma,z},k_{\gamma,z}\rangle_\gamma\notag\\
                        &=\frac{1}{K_\gamma(z,z)}\langle B_{\alpha,\gamma} K_{\gamma,z},K_{\gamma,z}\rangle_\gamma\\
                        &=h(z,z)^{N+\gamma-\alpha}.
                        \end{split}\end{equation}
                        It completes the proof.
\end{proof}

{\noindent\bf Proof of Theorem \ref{thms}.}  It follows from Lemma \ref{poisp} and Remark \ref{preka}  that $B_{\alpha,\gamma}:A^2(dv_\gamma)\rightarrow A^2(dv_\gamma)$ is a prediagonal operator with characteristic $\{\frac{(\alpha)_{\bm{m}}}{(N)_{\bm{m}}}\},$ where of course counting multiplicities. Since $A^2(dv_\gamma)$ is a closed subspace of the Hilbert space $L^2(dv_\gamma),$ combing with (\ref{oid}) implies that $B_{\alpha,\gamma}:L^2(dv_\gamma)\rightarrow L^2(dv_\gamma)$ is a prediagonal operator with characteristic  $\{\frac{(\alpha)_{\bm{m}}}{(N+\gamma)_{\bm{m}}}\}\cup\{0\}$. Then Lemma \ref{dgf} yields that $B_{\alpha,\gamma}\in S_p(L^2(dv_\gamma))$ if and only if  $B_{\alpha,\gamma}\in S_p(A^2(dv_\gamma))$ if and only if $\{\frac{(\alpha)_{\bm{m}}}{(N+\gamma)_{\bm{m}}}\}\in \ell^p$ for $0<p<\infty.$ Thus (1) is equivelent to (2) which is equivalent to $\{\frac{(\alpha)_{\bm{m}}}{(N+\gamma)_{\bm{m}}}\}\in \ell^p.$  It remains to prove $\{\frac{(\alpha)_{\bm{m}}}{(N+\gamma)_{\bm{m}}}\}\in \ell^p $ is equivalent to (4) and (3)  is equivalent to (4) in the condition $\alpha\in\mathscr{B}_\gamma.$  We first prove the equivalence of (3) and (4). From Lemma \ref{bere},  we know that $\widetilde{B_{\alpha,\gamma}}\in L^p(d\lambda)$ if and only if $$\int_{\Omega}\vert\widetilde{B_{\alpha,\gamma}}\vert^p d\lambda=\int_{\Omega}h^{p(N+\gamma-\alpha)-N}dv<\infty.$$ This along with the well-known fact that
$\int_{\Omega}h^tdv<\infty$ for $t\in\mathbb{R}$  is equivalent to $t>-1$ implies that $p>\frac{1+b+(r-1)a}{N+\gamma-\alpha}=\frac{N-1}{N+\gamma-\alpha},$ which proves  the equivalence of (3) and (4).

Now we turn to prove $\{\frac{(\alpha)_{\bm{m}}}{(N+\gamma)_{\bm{m}}}\}\in \ell^p$ for $0<p<\infty$  is equivalent to  $p>\frac{N-1}{N+\gamma-\alpha}.$ Suppose $p>\frac{N-1}{N+\gamma-\alpha}=\frac{1+b+(r-1)a}{N+\gamma-\alpha},$ namely $p(N+\gamma-\alpha)-(r-1)a-b>1,$ so $$p(N+\gamma-\alpha)-(r-j)a-b>1,j=1,\cdots,r.$$ Then (\ref{inp}),(\ref{asyy}) and Lemma \ref{dime} imply that \begin{equation}\begin{split}
\sum_{\bm{m}\in I(\Omega)}&\left\vert\frac{(\alpha)_{\bm{m}}}{(N+\gamma)_{\bm{m}}}\right \vert^pdim\hspace{1mm}\mathcal{P}_{\bm{m}}(\mathfrak{p}^+)\\
&=\left\vert\frac{(\alpha)_{\bm{0}}}{(N+\gamma)_{\bm{0}}}\right\vert^p dim\hspace{1mm}\mathcal{P}_{\bm{0}}(Z)+\sum_{k=1}^r\sum_{\bm{m}\in I_k(\Omega)}\left\vert\frac{(\alpha)_{\bm{m}}}{(N+\gamma)_{\bm{m}}}\right\vert^p dim\hspace{1mm}\mathcal{P}_{\bm{m}}(\mathfrak{p}^+)\notag\\
&\leq \left\vert\frac{(\alpha)_{\bm{0}}}{(N+\gamma)_{\bm{0}}}\right\vert^p dim\hspace{1mm}\mathcal{P}_{\bm{0}}(Z)+C\sum_{k=1}^r\sum_{\bm{m}\in I_k(\Omega)}\prod_{j=1}^k\frac{1}{m_j^{p(N+\gamma-\alpha)-(r-j)a-b}}\\
&\leq \left\vert\frac{(\alpha)_{\bm{0}}}{(N+\gamma)_{\bm{0}}}\right\vert^p dim\hspace{1mm}\mathcal{P}_{\bm{0}}(Z)+C\sum_{k=1}^r\prod_{j=1}^k\left(\sum_{m=1}^\infty\frac{1}{m^{p(N+\gamma-\alpha)-(r-j)a-b}}\right)\\
&<\infty,
\end{split}\end{equation}
which means that $\{\frac{(\alpha)_{\bm{m}}}{(N+\gamma)_{\bm{m}}}\}\in \ell^p$ for $0<p<\infty.$

Suppose $\{\frac{(\alpha)_{\bm{m}}}{(N+\gamma)_{\bm{m}}}\}\in \ell^p$ for $0<p<\infty.$  Then (\ref{inp}) and Lemma \ref{dime} imply that
\begin{equation}\begin{split}
\infty&>\sum_{\bm{m}\in I(\Omega)}\left\vert\frac{(\alpha)_{\bm{m}}}{(N+\gamma)_{\bm{m}}}\right\vert^p dim\hspace{1mm}\mathcal{P}_{\bm{m}}(\mathfrak{p}^+)\notag\\
&=\left\vert\frac{(\alpha)_{\bm{0}}}{(N+\gamma)_{\bm{0}}}\right\vert^p dim\hspace{1mm}\mathcal{P}_{\bm{0}}(Z)+\sum_{k=1}^r\sum_{\bm{m}\in I_k(\Omega)}\left\vert\frac{(\alpha)_{\bm{m}}}{(N+\gamma)_{\bm{m}}}\right\vert^p dim\hspace{1mm}\mathcal{P}_{\bm{m}}(\mathfrak{p}^+)\\
&> \sum_{\bm{m}\in I_1(\Omega)}\left\vert\frac{(\alpha)_{\bm{m}}}{(N+\gamma)_{\bm{m}}}\right\vert^p dim\hspace{1mm}\mathcal{P}_{\bm{m}}(\mathfrak{p}^+)\\
&\geq \frac{1}{C}\sum_{m=1}^\infty\frac{1}{m^{p(N+\gamma-\alpha)-(r-1)a-b}},\\
\end{split}\end{equation}
which implies that $p(N+\gamma-\alpha)-(r-1)a-b>1,$ namely $p>\frac{1+b+(r-1)a}{N+\gamma-\alpha}=\frac{N-1}{N+\gamma-\alpha}.$ \qed
\section{trace formulae for Bergman-type operators}
From Theorem \ref{thmf},\ref{thms} we know that $B_{\alpha,\gamma}$ belongs to the trace class $S_1(L^2(dv_\gamma))$ or $S_1(A^2(dv_\gamma))$ if and only if $\alpha<1+\gamma$ or $\alpha\in\mathscr{F}.$ Since Bergman space $A^2(dv_\gamma)$ is a closed subspace of  $L^2(dv_\gamma) $ and (\ref{ka0}),  it follows that  the trace $Tr(B_{\alpha,\gamma})$ on $L^2(dv_\gamma)$ and  $A^2(dv_\gamma)$ are the same when $\alpha<1+\gamma$ or $\alpha\in\mathscr{F}.$ Moreover,  $B_{\alpha,\gamma}$ belongs to the Hilbert-Schmidt class $S_2(L^2(dv_\gamma))$ or $S_2(A^2(dv_\gamma))$ if and only if $\alpha<\frac{N+1+2\gamma}{2}.$  Similarly,  the trace $Tr(B_{\alpha,\gamma}^\ast B_{\alpha,\gamma})$  on $L^2(dv_\gamma)$ and  $A^2(dv_\gamma)$ are the same when  $\alpha<\frac{N+1+2\gamma}{2}.$

\begin{prop} \label{kss} (1) The operator $B_{\alpha,\gamma}\in S_1$ if and only if $\alpha<1+\gamma$ or $\alpha\in\mathscr{F}.$ In this case, $$Tr(B_{\alpha,\gamma})=\int_{\Omega}h(z,z)^{-\alpha}dv_\gamma(z).$$

(2) The operator $B_{\alpha,\gamma}\in S_2$ if and only if $\alpha<\frac{N+1+2\gamma}{2}.$ In this case, $$Tr(B_{\alpha,\gamma}^\ast B_{\alpha,\gamma})=\int_{\Omega\times\Omega} \vert h(w,z)\vert ^{-2\alpha}dv_\gamma(w)dv_\gamma(z).$$

\end{prop}
\begin{proof} (1) It suffices to prove the trace formula.  It implies from (\ref{kf0}) that the operator $B_{\alpha,\gamma}:A^2(dv_\gamma)\rightarrow A^2(dv_\gamma)$ is a prediagonal operator.  Let $\{\lambda_n\}$ be its characteristic  and  $\{e_n\}$ be an orthogonal base of $A^2(dv_\gamma) $  such that $B_{\alpha,\gamma} e_n=\lambda_n e_n $ for each $n.$ Since $B_{\alpha,\gamma}\in S_1,$ it follows from Lemma \ref{dgf} that $\{\lambda_n\}\in l^1.$  Then the dominated convergence theorem  implies  that
\begin{equation}\begin{split}
Tr(B_{\alpha,\gamma})&=\sum\langle B_{\alpha,\gamma} e_n,e_n\rangle_\gamma \notag\\
                   &=\int_{\Omega}\sum \langle B_{\alpha,\gamma} e_n,K_{\gamma,z}\rangle_\gamma\bar{e}_n(z) dv_\gamma(z)\\
                   &=\int_{\Omega}\sum \langle e_n \bar{e}_n(z) ,B_{\alpha,\gamma}^\ast K_{\gamma,z}\rangle_\gamma dv_\gamma(z)\\
                   &=\int_{\Omega}\langle \sum e_n \bar{e}_n(z) ,B_{\alpha,\gamma}^\ast K_{\gamma,z}\rangle_\gamma dv_\gamma(z)\\
                   &=\int_{\Omega}\langle B_{\alpha,\gamma} K_{\gamma,z} , K_{\gamma,z}\rangle_\gamma dv_\gamma (z).\\
 \end{split}\end{equation}
 This along with  and (\ref{kkz}) shows that $$Tr(B_{\alpha,\gamma})=\int_{\Omega}h(z,z)^{-\alpha}dv_\gamma(z).$$

(2) Similarly, we have
\begin{equation}\begin{split}
Tr(B_{\alpha,\gamma}^\ast B_{\alpha,\gamma})&=\sum\langle B_{\alpha,\gamma}^\ast B_{\alpha,\gamma} e_n,e_n\rangle_\gamma\notag\\
             &=\int_{\Omega}\sum \langle B_{\alpha,\gamma}^\ast  B_{\alpha,\gamma} e_n,K_z\rangle_\gamma \bar{e}_n(z) dv_\gamma(z)\\
                   &=\int_{\Omega}\sum \langle e_n \bar{e}_n(z) ,B_{\alpha,\gamma}^\ast  B_{\alpha,\gamma} K_z\rangle_\gamma dv_\gamma(z)\\
                   &=\int_{\Omega}\langle  B_{\alpha,\gamma} K_z , B_{\alpha,\gamma} K_z\rangle_\gamma dv_\gamma(z)\\
                  &= \int_{\Omega\times\Omega} \vert h(w,z)\vert ^{-2\alpha}dv_\gamma(w)dv_\gamma(z).
 \end{split}\end{equation}
\end{proof}

\begin{rem}  In fact, the trace $Tr(B_{\alpha,\gamma})$ can be exactly calculated as follow. $$Tr(B_{\alpha,\gamma})=\frac{c_\gamma}{c_{\gamma-\alpha}}=\frac{\Gamma_{\Omega}(N+\gamma)}{\Gamma_\Omega(N+\gamma-\alpha)}\frac{\Gamma_\Omega(\frac{a}{2}(r-1)+1+\gamma-\alpha)}{\Gamma_{\Omega}(\frac{a}{2}(r-1)+1+\gamma)}.$$
\end{rem}
We first note that $Tr(B_{\alpha,\gamma}^\ast B_{\alpha,\gamma})<\infty $ if and only if $\alpha<\frac{N+1+2\gamma}{2}.$ Aa an application, we establish  an integral   estimate  related to the Forelli-Rudin estimate on the bounded symmetric domain. Denote $J_{\beta,\gamma}$ on $\Omega$ by
$$J_{\beta,\gamma}(z)=\int_\Omega\frac{h(w,w)^{\gamma}}{\vert h(z,w)\vert^{N+\beta+\gamma}}dv(w), \quad z\in\Omega, $$
 where of course we assume $\gamma>-1.$ The asymptotic estimation of $J_{\beta,\gamma}$  for $\vert \beta\vert >\frac{a}{2}(r-1)$ have been obtained in \cite{FK}; see \cite{EnZ,Yan} for the gap  $\vert \beta\vert \leq\frac{a}{2}(r-1)$ in some spacial cases.
\begin{cor} \label{FRe}The integral  $\int_{\Omega}J_{\beta,\gamma}dv<\infty$ if and only if $\beta<1.$
\end{cor}
\begin{proof} It follows from \cite[Corollary 3.7 and Theorem 3.8]{FK} that
\begin{equation}\begin{split}\label{jgz}
J_{\beta,\gamma}(z)&=\int_\Omega\frac{h(w,w)^{\gamma}}{\vert h(z,w)\vert^{N+\beta+\gamma}}dv(w)\\
                             &=\sum_{\bm{m}\geq0}\frac{\vert(\frac{N+\beta+\gamma}{2})_{\bm{m}}\vert^2}{(\gamma+N)_{\bm{m}}}K^{\bm{m}}(z,z)\\
                             &=\sum_{k=0}^r\sum_{\bm{m}\in I_\Omega(k)}\frac{\vert(\frac{N+\beta+\gamma}{2})_{\bm{m}}\vert^2}{(\gamma+N)_{\bm{m}}}K^{\bm{m}}(z,z)\\
\end{split}\end{equation}
 On the other hand, Stirling's formula implies that there exists a positive constant $C$ such that $$\frac{1}{ C}  \frac{\vert(\frac{N+\beta}{2})_{\bm{m}}\vert^2}{\vert (N)_{\bm{m}}\vert }\leq\frac{\vert(\frac{N+\beta+\gamma}{2})_{\bm{m}}\vert^2}{\vert (\gamma+N)_{\bm{m}}\vert }\leq C  \frac{\vert(\frac{N+\beta}{2})_{\bm{m}}\vert^2}{\vert (N)_{\bm{m}}\vert }
$$ for any $\bm{m}\geq0.$ Combing this with (\ref{jgz}) shows that  the integral  $\int_{\Omega}J_{\beta,\gamma}dv<\infty$ if and only if the integral  $\int_{\Omega}J_{\beta,0}dv<\infty.$ Then Proposition \ref{kss} implies that the integral  $\int_{\Omega}J_{\beta,0}dv<\infty$ if and only  if $\beta<1.$
It completes the proof.
\end{proof}
In the case of Hilbert unit ball, Corollary \ref{FRe} is immediate from the  Forelli-Rudin estimate \cite[Theorem 1.12]{Zhu3}. To end this section, we give a compactness result on $B_{\alpha,\gamma}.$

\begin{cor}\label{Bcpt} If $\alpha<N+1+2\gamma,$ then $B_{\alpha,\gamma}: L^\infty(dv_\gamma)\rightarrow L^1(dv_\gamma)$ is compact.
\end{cor}

To prove the above corollary, we need the following lemma  which provides a criteria of precompactness in $L^p(dv_\gamma)$ with $\gamma>-1$ and $1\leq p<\infty,$ whose proof is similar to  \cite[Theorem 2.33]{AdF}. Recall that a subset in a Banach space is called precompact if its closure  in the norm topology is compact. Obviously, an operator between two Banach spaces is compact if and only if the operator maps every bounded subset to precompact one; in particular, the compact operators are all bounded.

\begin{lem}\cite{AdF} \label{AdF} Let $1\leq p<\infty$ and $E\subset L^p(dv_\gamma).$ Suppose there exists a sequence $\{\Omega_j\}$ of subdomains of $\Omega$ having the following properties:
\begin{enumerate}
\item  $\Omega_j\subset\Omega_{j+1}.$
\item The set of restrictions to $\Omega_j$ of the functions in $E$ is precompact in  $L^p(\Omega_j,dv_\gamma)$ for each $j.$
\item For every $\varepsilon>0$ there exists a $j$ such that $$ \int_{\Omega-\Omega_j}\vert f\vert^pdv_\gamma<\varepsilon,\quad\forall f \in E.$$
\end{enumerate}
Then $E$ is precompact in $ L^p(dv_\gamma).$
\end{lem}

 {\noindent{\bf{Proof of  Corollary \ref{Bcpt}.}}
 Suppose $\{f_j\}$ is  an arbitrary bounded sequence in $L^\infty(dv_\gamma),$ without loss of generality, we can suppose that
   \begin{equation}\label{bds} \Vert f_j\Vert_{L^\infty} \leq 1,\quad j=1,2,\cdots.\notag\\ \end{equation}
   Denote the  bounded domain $B_j'$ by $B_j'=\{z\in\mathbb{C}^d:\vert z\vert_{sp}<1-\frac{1}{j}\},j=1,2,\cdots.$ Clearly,   $B_j'$ is compactly contained in $\Omega$ and $B_j'\subset B_{j+1}'\subset \Omega,$ for every $j.$ We first prove that the  set of restrictions to $B_j'$ of the functions in $\{B_{\alpha,\gamma}f_n\}$ is precompact in  $L^p(B_j')$ for each $j.$ In view of Lemma \ref{hol}, we know the functions in $\{B_{\alpha,\gamma}f_n\}$ are all continuous on  $\Omega$ and uniformly  continuous  on every $\overline{B_j'}.$ Combing with the fact that the embedding $C(\overline{B_j'})\subset L^p(B_j',dv_\gamma)$ is continuous for every $j,$ it is enough to prove $\{(B_{\alpha,\gamma}f_n)|_{B_j'}\}$ is precompact in $C(\overline{B_j'})$ for every $j.$
Note that $h(z,w)^{-\alpha}$ is holomorphic on $\overline{B_j'}\times\overline{\Omega}\subset\Omega\times\overline{\Omega},$ so  $h(z,w)^{-\alpha}$ is uniformly  continuous   on $\overline{B_j'}\times\overline{\Omega}$ and there exists $C_j>0$ such that $$\vert h(z,w)^{-\alpha}\vert \leq C_j$$ on $\overline{B_j'}\times\overline{\Omega}$ for every $j.$  We then have that
\begin{equation}\begin{split}\label{Kaf}
\Vert ( B_{\alpha,\gamma} f_n)|_{B_j'}\Vert_{L^\infty}&=\sup_{z\in \overline{B_j'}}\vert B_{\alpha,\gamma} f_n\vert \\
&=\sup_{z\in \overline{B_j'}}\left\vert\int_{\Omega}\frac{f_n(w)dv_\gamma(w)}{h(z,w)^\alpha}\right\vert\\
 &\leq C_j\Vert f_n\Vert_{L^\infty}\\
&\leq C_j
\end{split}\end{equation} for every $j.$ The estimate (\ref{Kaf}) implies that  $\{(B_{\alpha,\gamma}f_n)|_{B_j'}\}$ are uniformly bounded in $C(\overline{B_j'})$ for every $j.$ Meanwhile, the uniform continuity of the function  $h(z,w)^{-\alpha}$   on $\overline{B_j'}\times\overline{\Omega}$ implies that $\{(B_{\alpha,\gamma}f_n)|_{B_j'}\}$ is equicontinuous on $\overline{B_j'}$ for every $j.$ Then  Arzel\`a-Ascoli theorem implies that $\{(B_{\alpha,\gamma}f_n)|_{B_j'}\}$ is  precompact in $C(\overline{B_j'})$ for every $j.$ Due to Proposition \ref{kss}
 we obtain  that, for any fixed $\alpha<N+1+2\gamma$ and for any $\varepsilon>0,$ there exists a $J>0$ satisfying \begin{equation}\label{ie} \int_{\Omega-B_j'} dv_\gamma(z)\vert h(z,w)\vert ^{-\alpha}dv_\gamma(w) <\varepsilon,\quad \forall j>J,\end{equation}  since  the absolute continuity of the integral and $\lim_{j\rightarrow\infty}v_\gamma(\Omega-B_j')=0.$  It follows from (\ref{ie}) that
 \begin{equation}\begin{split}
\int_{\Omega-B_j'}\vert B_{\alpha,\gamma}f_n(z)\vert dv_\gamma(z)&=\int_{\Omega-B_j'}\left\vert\int_{\Omega}\frac{f_n(w)dv_\gamma(w)}{h(z,w)^\alpha}\right\vert dv_\gamma(z)\notag\\
&\leq \Vert f_n\Vert_{L^\infty}\int_{\Omega-B_j'} dv_\gamma(z) \int_{\Omega}\vert h(z,w)^{-\alpha}\vert dv_\gamma(w)\\
&\leq \varepsilon,\\
\end{split}\end{equation}
for any $j>J$ and $n=1,2,\cdots.$ Combing with  Lemma \ref{AdF} it implies that  $\{B_{\alpha,\gamma}f_n\}$ is precompact in $L^1(dv_\gamma).$   Thus $B_{\alpha,\gamma}: L^\infty(dv_\gamma)\rightarrow L^1(dv_\gamma)$ is compact. It completes the proof.
\qed

\section{Schatten class Szeg\"o-type  operators}

We have  characterized the Schatten class  Bergman-type operators on the Bergman space $A^2(dv_\gamma)$ so far. The above argument can be used for the  characterization of Schatten class of the so-called Szeg\"o-type integral operators on Hardy spaces.  Let $S$ be the Shilov boundary of $\Omega$ with normalized $K$-invariant Haar measure $d\sigma,$  the Hardy space $H^2(\Omega)$ is consists of  all holomorphic functions on $\Omega$ satisfying 
$$\sup_{R\rightarrow 1^-}\int_{S}\vert f(Rz)\vert^2d\sigma(z)<\infty.$$
  There exists a canonical unitary  isomorphism between $H^2(\Omega)$ and a closed subspace $H^2(S)$ of  $L^2(S,d\sigma);$  see \cite{Kor,Up,UpW} for details
about Hardy spaces. Thus we can identify  $H^2(\Omega)$ with  $H^2(S).$ Szeg\"o kernel of $\Omega$ is given by
$$S(z,\xi)=h(z,\xi)^{-\rho}, \quad (z,\xi)\in \Omega\times S,$$ where $\rho:=1+\frac{a}{2}(r-1)+b.$
\begin{defi} For $\alpha\in \mathbb{R},$  the Szeg\"o-type integral operator $H_\alpha$ on $L^1(S,d\sigma)$ is defined by
$$ H_\alpha f(z)=\int _{S}\frac{f(\xi)}{h(z,\xi)^\alpha}d\sigma(\xi),\quad z\in\Omega.$$
\end{defi}
Note that $H_\rho$ is the Szeg\"o projection  from $L^2(S)$ onto $H^2(\Omega).$ The same  argument in Lemma \ref{bksd} shows that $ H_\alpha:L^2(S)\rightarrow L^2(S)$ is bounded if and only if $ H_\alpha:H^2(\Omega)\rightarrow H^2(\Omega)$ is bounded if and only if $\alpha\leq \rho.$ Denote the set $\mathscr{S}$  by $$
\mathscr{S}:=\{s< \rho:s\notin\mathscr{F}\}. $$

Now we list the main results of Schatten class of Szeg\"o-type  operators, the proof is similar to the case of Bergman-type case and omitted.
\begin{thm}\label{thmff} If $\alpha\in\mathscr{F},$ then the followings hold.
\begin{enumerate}
\item  The operator $H_\alpha\in S_p(L^2(S))$ for any $p>0.$
\item   The operator $H_\alpha\in S_p(H^2(\Omega))$ for any $p>0.$
\end{enumerate}
\end{thm}

\begin{thm}\label{thmss} Suppose $ \alpha\in\mathscr{S} $ and $0 < p<\infty,$  then the following statements are equivalent.
\begin{enumerate}
\item $H_\alpha\in S_p(L^2(S)).$
\item $H_\alpha\in S_p(H^2(\Omega)).$
\item $\widetilde{H_\alpha}\in L^p(d\lambda).$
\item $p>\frac{N-1}{\rho-\alpha}.$
\end{enumerate}
 \end{thm}
Where  $\widetilde{H_\alpha}$ is the Berezin transform of $H_\alpha$ on Hardy space $H^2(\Omega).$

\bibliographystyle{plain}

\begin{thebibliography}{99}
\small
\bibitem{AdF}
A. Adama, F. Fournier, \emph{ Sobolev Spaces}, Pure and Applied Mathematics (Amsterdam), Elsevier/Academic Press, Amsterdam (2003)
\bibitem{BB}
D. B\'ekoll\'e, A. Bonami, \emph{ Estimates for the Bergman and Szeg\"o projections in two symmetric domains of $\mathbb{C}^n$,} Colloq. Math. \textbf{68}, 81-100 (1995)
\bibitem{BGN}
 A. Bonami, G. Garrig\'os,  C. Nana, \emph{$L^p$-$L^q$ estimates for Bergman projections in bounded symmetric domains of tube type,} J. Geom. Anal. \textbf{24}, 1737-1769 (2014)

\bibitem{CF}
G. Cheng, X. Fang, Z. Wang, J. Yu, \emph{The hyper-singular cousin of the Bergman projection,} Trans. Amer. Math. Soc. \textbf{369}, 8643-8662 (2017)

\bibitem{Car}
E. Cartan, \emph{Sur les domaines born\'es homog\'enes de l\'espace den variables complexes} (French), Abh. Math. Sem. Univ. Hamburg, \textbf{11}, 116-162 (1935)
\bibitem{CH}
G. Cheng, X. Hou, C. Liu, \emph{The singular integral operator induced by Drury-Arveson kernel}, Complex Anal. Oper. Theory, \textbf{12}, 917-929 (2018)
\bibitem{Ding}
L. Ding, J. Fan, \emph{On the  compactness  of Bergman-type  integral  operators,} arXiv:2004.13635, (2020)
\bibitem{DiW}
L. Ding, K. Wang, \emph{The  $L^p$-$L^q$ problems of Bergman-type operators,} arXiv:2003.00479, (2018)
\bibitem{EnZ}
M. Englis, G. Zhang, \emph{On the Faraut-Koranyi hypergeometric functions in rank two,} Ann. Inst. Fourier (Grenoble), \textbf{54,} 1855-1875 (2005)
\bibitem{FK}
J. Faraut, A.Kor\'anyi,\emph{ Function spaces and reproducing kernels on bounded symmetric domains, } J. Funct. Anal., \textbf{88,} 64-89 (1990)
\bibitem{FR}
F. Forelli, W.  Rudin, \emph{Projections on spaces of holomorphic functions in balls,} Indiana Univ. Math. J. \textbf{24}, 593-602 (1974)
\bibitem{Hel}
 S. Helgason, \emph{Differential Geometry, Lie Groups, and Symmetric Spaces,}  Academic Press, New York, (1978)
 \bibitem{JaP}
M. Jarnicki, P. Pflug, \emph{Invariant distances and metrics in complex analysis, Second extended edition,} De Gruyter Expositions in Mathematics, \textbf{9,} Walter de Gruyter,  Berlin (2013)
 \bibitem{KU}
 H.  Kaptanoglu, A. Ureyen, \emph{ Singular integral operators with Bergman-Besov kernels on the ball,}
Integral Equations Operator Theory, \textbf{91}, No. 30 (2019)
\bibitem{Kor}
A. Kor\'anyi, \emph{The Poisson integral for generalized half-planes and bounded symmetric domains,} Ann. Math. (2), \textbf{82}, 332-350 (1965)

\bibitem{Loo}
O. Loos, \emph{Jordan Pairs,} Lecture Notes in Math., \textbf{460,} Springer (1975)
\bibitem{MP}
S. Miihkinen, J. Pau; A. Per\"al\"a,  M. Wang,  \emph{ Volterra type integration operators from Bergman spaces to Hardy spaces}, J. Funct. Anal. \textbf{279}, 32 pp  (2020)
\bibitem{PaP}
J, Pau, Jordi, A. Perala, \emph{ A Toeplitz-type operator on Hardy spaces in the unit ball,} Trans. Amer. Math. Soc.,\textbf{373}, 3031-3062 (2020)
 \bibitem{Rin}
J. Ringrose, \emph{ Compact non-self-adjoint operators,} Van Nostrand Reinhold Co., London  (1971)
 \bibitem{Say}
 I. Satake,  \emph{Algebraic structures of symmetric domains,} Iwanami Shoten, Tokyo, and Princeton Univ. Press, Princeton, NJ, (1980)
 \bibitem{Up}
 H. Upmeier, \emph{Toeplitz operators on bounded symmetric domains,} Trans. Amer. Math. SOC. \textbf{280}, 221-237  (1983)
 \bibitem{UpW}
 H. Upmeier, K. Wang, \emph{Dixmier trace for Toeplitz operators on symmetric domains,} J. Funct. Anal., \textbf{271,} 532-565 (2016)
\bibitem{Yan}
Z. Yan,\emph{ A class of generalized hypergeometric functions in several variables,} Canad. J. Math., \textbf{44,} 1317-1338 (1992)
\bibitem{Zha}
R. Zhao, \emph{Generalization of Schur's test and its application to a class of integral operators on the unit ball of $\mathbb{C}^n$,} Integral Equations Operator Theory, \textbf{82,} 519-532 (2015)
\bibitem{ZZ}
R. Zhao, K. Zhu, \emph{Theory of Bergman Spaces in the Unit Ball of $\mathbb{C}^n$,} M\'em. Soc. Math. Fr.  \textbf{115} (2009)
\bibitem{Zhu1}
K. Zhu, \emph{Holomorphic Besov spaces on bounded symmetric domains,} Quart. J. Math. Oxford, \textbf{46,} 239-256 (1995)
\bibitem{Zhu2}
K. Zhu, \emph{Holomorphic Besov spaces on bounded symmetric domains. II,} Indiana Univ. Math. J., \textbf{44,} 1017-1031 (1995)
\bibitem{Zhu3}
K. Zhu, \emph{ Spaces of Holomorphic Functions in the Unit Ball}, Graduate Texts in Mathematics, \textbf{226}, Springer-Verlag, New York (2005)
\bibitem{Zhu4}
K. Zhu, \emph{Operator Theory in Function Spaces. Second Edition}, Mathematical Surveys and Monographs, 138, American Mathematical Society, Providence, RI ( 2007)
\end{thebibliography}
 
\end{document}